\documentclass[leqno]{article}
\usepackage{amsmath}
\usepackage{amsthm}
\usepackage{amssymb}
\usepackage{graphicx} 
\usepackage{xcolor}
\usepackage{tikz}
\usepackage{url}
\usepackage{hyperref}
\usepackage{caption}
\usetikzlibrary{angles, quotes, intersections, calc}
\usepackage{pgfplots}
\pgfplotsset{compat=1.18}

\numberwithin{equation}{section} 
\newtheorem{theorem}[equation]{Theorem}
\theoremstyle{definition}

\theoremstyle{remark}

\theoremstyle{plain}
\newtheorem{lemma}[equation]{Lemma}
\newtheorem{corollary}[equation]{Corollary}
\newtheorem{proposition}[equation]{Proposition}

\title{Remarks on the relative isoperimetric profile of polygonal domains in $\mathbb{R}^2$}
\author{Jason DeVito, Robert DeYeso III,\\ Ezra Nance, Robert Niedzialomski}
\date{}

\begin{document}

\maketitle

\begin{abstract}
We develop techniques for solving the relative isoperimetric problem on polygonal domains in $\mathbb{R}^2$, with special attention paid to corners.  As an application, we solve the relative isoperimetric problem for a square with a square corner removed.
\end{abstract}

\section{Introduction}

Let $Q\subseteq \mathbb{R}^n$ be a measurable subset with non-zero measure $|Q|$. For any measurable subset $S\subseteq Q$, we let $P(S)$ denote the \textit{relative perimeter} of $S$, that is, the perimeter of $S$ in the interior of $Q$.  For each positive number $t\leq |Q|$, we define $f_{Q}(t) := \inf \{P(S): S\subseteq Q \text{ and } |S| = t\}$.  The function $f_Q$ is called the \textit{relative isoperimetric profile} of $Q$, and any subset $S$ which achieves this infimum is called an \textit{isoperimetric minimizer}.

Several necessary conditions are known about these minimizing subsets. For instance, any minimizer which intersects the boundary of $Q$ must do so orthogonally. However, orthogonality loses all meaning if $\partial Q$ is not smooth at the intersection point. Since these non-smooth points cause complications, often smoothness conditions are assumed on the boundary of $Q$, see for example \cite{CGR,MBNPR,SZ}. 

Our primary goal in this article is to work towards removing the smoothness assumption on $\partial Q$. Specifically, we focus on the case when $Q$ is a polygonal domain, that is, a subset of $\mathbb{R}^2$ bounded by a piecewise linear curve. In Section~ \ref{sec:tools} we characterize all possible ways that an isoperimetric minimizer can interact with corners of a polygonal domain. Our main result is the following.

\begin{theorem}\label{thm:main2}
Let $Q\subseteq \mathbb{R}^2$ be a polygonal region and let $K\in Q$ be a corner. Suppose $S$ is an isoperimetric minimizer. \begin{enumerate}\item If $K$ is convex, then $\partial S$ does not contain $K$.

\item  If $K$ is non-convex, then at most one smooth boundary component of $\partial S$ contains $K$, and this smooth boundary component does not make an acute angle with either of the two edges of $Q$ which contain $K$.
\end{enumerate}
\end{theorem}

The rest of this article is dedicated to the study of a specific family of polygonal domains. We consider the 1-parameter family of the unit square with a square corner removed, that is, $Q_a := [0,1]^2 \setminus [0,a)^2$. These regions were chosen because they offer a relatively simple example of a polygonal domain with a non-convex corner. Furthermore, when $a=0$ we obtain the full unit square, which has already been analyzed in \cite{BB}. This article generalizes their result.
 
Let $f_a$ be shorthand for $f_{Q_a}$. In \cite{BB}, descriptions of isoperimetric minimizing regions were provided as well as an explicit formula for $f_0$. We do the same for $f_a$ for all $a \in [0,1)$. Our main result is the following.

\begin{theorem}\label{thm:main}For each $a\in [0,1)$ and $t\leq|Q_a|$, isoperimetric minimizers are given by one of the regions $S_1,S_2,S_3$ or $S_4$ in Figure \ref{fig:main}.  An explicit formula for $f_a$ is given by equation \eqref{eq:profile}.
\end{theorem}

It is interesting that $f_a$ goes through phase changes as $a$ increases. For small $a$, only $S_1$ and $S_2$ regions arise, that is, $f_a$ agrees with the formula found in \cite{BB}. However, once $a$ becomes larger, we must also consider $S_4$ regions, which can intersect the non-convex corner at unusual angles.  For even larger values of $a$, $S_3$ regions become relevant.

It follows from our proof that every isoperimetric minimizer must be a region of the form $S_1,S_2,S_3$, or $S_4$.  In addition, for each $a$ and all but finitely many $t$, it follows from our proof that a minimizing region $S$ is unique up to the obvious symmetries, e.g, a quarter circle can be placed in any of the five convex corners. 

As an important corollary, this example shows that the boundary of a minimizer can intersect a non-convex corner, so that Theorem~\ref{thm:main2}(2) is sharp.

\begin{center}
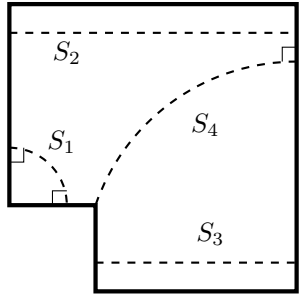

    \begin{tikzpicture}[scale=3.8]
        \draw[ultra thick] (0.3,0) -- (1,0) -- (1,1) -- (0,1) -- (0,0.3) -- (0.3,0.3) -- cycle;
        \draw[thick, dashed] (0,.9) -- (1,.9);
        \draw[thick, dashed] (.3,0.1) -- (1,0.1);
        \draw[thick, dashed] (1.0,0.8) arc (90:161:0.74);
        \draw (0,0.45) -- (.05,0.45) -- (0.05, .5);
        \draw (.15,.3) -- (.15,.35) -- (.2,.35);
        \draw[thick, dashed] (0,.5) arc (90:0:.2);
        \draw (1,0.85) -- (0.95,0.85) -- (0.95,0.8);
        \draw (0.60,0.65) node[anchor=north west] {$S_4$};
        \draw (.1,.45) node[anchor = south west] {$S_1$};
        \draw (.2, .9) node[anchor = north] {$S_2$};
        \draw (.7, .13) node [anchor = south] {$S_3$};
    \end{tikzpicture}
    \captionof{figure}{Region $S_1$ is a quarter circle hitting both boundary components of $Q_a$ perpendicularly.  Regions $S_2$ and $S_3$ are bounded by line segments parallel to the sides of $Q_a$, and the boundary of Region $S_4$ is a portion of a circle connecting the non-convex corner to a side of length 1.}\label{fig:main}
\end{center}

\textbf{Acknowledgements}  The first author was supported by the NSF through DMS-2405266.  He is grateful for the support.

\section{Background}\label{sec:background}

Suppose $Q \subseteq \mathbb{R}^2$ is a polygonal region, i.e., a domain whose boundary is piecewise linear. Let $S\subseteq Q$ be a measurable region with area $|S|$. We denote by $P(S)$ the relative perimeter of $S$ in $Q$.  That is, $P(S)$ is defined by
\begin{equation*}
P(S)=\sup\left\{\int_S\textrm{div} (X)\,dx\colon X\in C^\infty_c(\text{int}(Q),\mathbb{R}^2),\,\|X\|\leq 1\right\},    
\end{equation*}
where $\text{int}(Q)$ is the interior of $Q$ and $\textrm{div} (X)$ is the divergence of the vector field~$X$.

We are interested in determining the \textit{relative isoperimetric minimizers} (or, more briefly, a \textit{minimizer}).  Recall that a minimizer is a measurable subset $S\subseteq Q$ for which  
\[
    P(S)\leq P(S') \text{ for all } S'\subseteq Q \text{ with } |S| = |S'|.  
\]
If $S$ is a minimizer, then its complement $Q \setminus S$ is also a minimizer bounding area $|Q| - |S|$.  In particular, we can always assume that $|S| \leq \frac{1}{2}|Q|$.

We have the following structure result for isoperimetric minimizers adapted to our situation.  The first bullet point  was proven by Gonzalez, Massari, and Tamanini \cite[Theorem 2, pg. 29]{GMT}, while the second two were proven by Grüter \cite[Theorem (iii) pg. 264]{Gr}.

\begin{theorem}
    \label{thm:structure}  
    Let $S\subseteq Q$ be an isoperimetric minimizer.  Then 
    \begin{itemize}
        \item $\partial S\cap \text{int}(Q)$ is an embedded analytic submanifold of $\text{int}(Q)$,
        \item  there is a number $c$ for which $\partial S\cap \text{int}(Q)$ is a disjoint union of curves, all of which have constant mean curvature $c$,
        \item  $\partial S$ intersects the non-corner points of $Q$ orthogonally.
    \end{itemize}
\end{theorem}

It is a priori possible that distinct connected components of $\partial S\cap \text{int}(Q)$ belong to the same connected component of $\partial S$, see Figure \ref{fig:TwoArcsFromCorner}. We define the number of \textit{smooth boundary curves of $S$} as the number of components of $\partial S\cap \text{int}(Q)$.  If no such curves touch an edge or corner of $Q$, this is simply the number of components of $\partial S$ as a topological manifold. In general, $S$ could have multiple smooth boundary components. 

\begin{center}
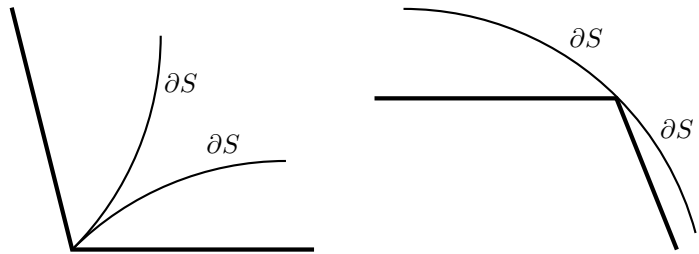

    \begin{tikzpicture}[scale=0.8]
        \draw[ultra thick] (4,0) -- (0,0) -- (-1,4);
        \draw[thick] (0,0) arc (-45:0:5);
        \draw[thick] (0,0) arc (135:90:5);
        \node at (2.5,1.75) {$\partial S$};
        \node at (1.8,2.75) {$\partial S$};
    \begin{scope}[shift={(5,0)}]
        \draw[ultra thick] (0,2.5) -- (4,2.5) -- (5,0);
        \draw[thick] (4.015,2.515) arc (45:90:5);
        \draw[thick] (4.015,2.515) arc (45:15:5);
        \node at (3.5, 3.5) {$\partial S$};
        \node at (5, 2) {$\partial S$};
    \end{scope}
    \end{tikzpicture}
    \captionof{figure}{Two examples showing that the boundary of $S$ is a connected topological manifold while $\partial S \cap \text{int}(Q)$ has two smooth components.}\label{fig:TwoArcsFromCorner}
\end{center}

We conclude this section with a proof that for a polygonal region $Q$ and for a minimizer $S$, the boundary $\partial S$ can have at most finitely many smooth connected components.

\begin{proposition}
\label{prop:finiteboundary}
Suppose $Q\subseteq\mathbb{R}^2$ is a polygonal region and $S \subseteq Q$ is relative isoperimetric minimizer.  Then $\partial S$ has finitely many smooth boundary components.
\end{proposition}

\begin{proof}

Suppose for a contradiction that $S$ is a minimizer for which $\partial S$ has infinitely many smooth boundary components.  From Theorem \ref{thm:structure}, the boundary components are either all line segments or all arcs of circles of some fixed radius~$r$.

Suppose initially that the boundary $\partial S$ consists of line segments.  We first claim that there are only finitely many smooth boundary components of $\partial S$ which intersect $Q$ in a corner.  This follows because such a line segment must either intersect one of the finitely many remaining corners or intersect $\partial Q$ perpendicularly, which determines the slope of the segment up to finitely many possibilities.  Since there are infinitely many smooth boundary components, there must be infinitely many which intersect two walls of $Q$ perpendicularly.  But there is a positive minimum such distance between parallel walls of $Q$, while there are infinitely smooth boundary components.  Thus, we find $P(S) = \infty$, a contradiction.

We next consider the case where there are infinitely many smooth boundary components of $\partial S$ all given by arcs of a circle of radius $r$.  For any two parallel walls of $Q$ there are no such arcs intersecting them orthogonally.  For any two non-parallel walls, in order for a component of $\partial S$ to intersect them both orthogonally, the center of this arc must lie on the lines formed by extending both walls.  Since the radius and center are determined, there is at most one such arc lying in $Q$.  Thus, by hypothesis, there must be infinitely many smooth components intersecting one corner $K$ of $Q$.  For each arc containing $K$, it either ends on a wall of $Q$ or another corner of $Q$.  But for each wall, there are at most two such arcs, because the center of the arc must lie in the line made by extending the wall and the distance to $K$ must be $r$.  And for each corner distinct from $K$ there are at most two arcs containing $K$ and this corner, because the center must simultaneously lie on a circle of radius $r$ about each corner points.  Thus, there are only finitely many arcs of circles of radius $r$ satisfying the condition that they meet walls of $Q$ orthogonally, contradicting the fact that there are infinitely many smooth components.

\bigskip

\end{proof}

\section{An analysis of corners}\label{sec:tools}  In this section, we carry out a detailed analysis of how the boundary curves of an isoperimetric minimizer can intersect corners of a polygonal region $Q\subseteq \mathbb{R}^2$.  In particular, we show that $\partial S$ cannot intersect a convex corner and that at most one smooth component of $\partial S$ can intersect each non-convex corner.  

\begin{theorem}\label{thm:corner}
Let $Q\subseteq \mathbb{R}^2$ be a polygonal region and let $K\in Q$ be a corner. Suppose $S$ is an isoperimetric minimizer. \begin{enumerate}\item If $K$ is convex, then $\partial S$ does not contain $K$.

\item  If $K$ is non-convex, then at most one smooth boundary component of $\partial S$ contains $K$, and this smooth boundary component does not make an acute angle with either of the two edges of $Q$ which contain $K$.
\end{enumerate}
\end{theorem}

When we compute the relative isoperimetric profile of the region $Q_a$, we will see that in some situations the boundary of an isoperimetric minimizer must intersect the non-convex corner of $Q_a$.  In particular, Theorem~\ref{thm:corner}(2) is sharp.  We will prove the theorem after establishing a number of preliminary propositions.  

\begin{proposition}\label{prop:circleTangentToSide} Let $Q\subseteq\mathbb{R}^2$ be a polygonal region and let $K\in Q$ be a convex corner.  Suppose that $S\subseteq Q$ has a unique smooth component containing $K$.  If this component is tangent to $Q$ at $K$, then $S$ is not an isoperimetric minimizer.
\end{proposition}

\begin{proof} By way of contradiction, assume that $S$ is an isoperimetric minimizer. We will find a new region $S'$ with $|S'| = |S|$ and with $P(S') < P(S)$.

Since $\partial S$ has constant mean curvature, the fact that it is tangent to $Q$ at $K$ implies that the smooth component of $\partial S$ containing $K$ is an arc of a circle.   Let $Q_1$ denote the side of $Q$ which is tangent to the arc.

\begin{center}
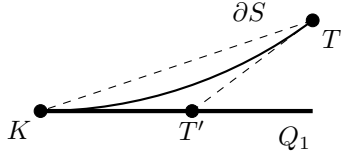

    \begin{tikzpicture}[scale=1.2]
        \draw[ultra thick] (0,-5) -- (3,-5);
        \filldraw (0,-5) circle (2pt) node[anchor=north east] {$K$};
        \draw (2.3,-3.9) node {$\partial S$};
        \draw (2.8,-5.3) node {$Q_1$};
        \draw[thick] (0,-5) arc (-90:-53.13:5);
        \draw[dashed] (3,-4) -- (0,-5);
        \draw[dashed] (3,-4) -- (1.67,-5) node[anchor=north] {$T'$};
        \filldraw (1.67,-5) circle (2pt);
        \filldraw (3,-4) circle (2pt) node[anchor=north west] {$T$};
    \end{tikzpicture}
    \captionof{figure}{Neighborhood of a corner $K$ with a circular arc.}\label{fig:arcTangentToCorner}
\end{center}

By Proposition \ref{prop:finiteboundary}, $\partial S$ has only finitely many smooth boundary components. Therefore, we may find a neighborhood of $K$, which intersects $\partial S$ only in a portion of this circular arc.  Given any point $T$ on this circular arc, we let $T'$ denote the intersection of the tangent line at $T$ and $Q_1$, see Figure \ref{fig:arcTangentToCorner}.  By choosing $T$ close enough to $K$ on the circular arc, we may assume that the angle $\angle TKT'$ is acute and the angle $
\angle TT'K$ is obtuse.

The secant line from $T$ to $K$ is shorter than the circular arc between these points.  Moreover, since the angle $\angle KT'T$ is obtuse, it follows that as we deform the secant line to the tangent line through a family of line segments, the lengths decrease.  In particular, they are all shorter than the circular arc.

By replacing $S$ with its complement, we assume without loss of generality that the region $S$ lies inside the circular arc.   It then follows that the secant line underestimates the area while the tangent line overestimates it.  In particular, by continuity, there is some line segment in the deformation whose area agrees with the area of $S$.  Thus, by replacing the arc by this line segment, we obtain the desired subset~$S'$.

\end{proof}

The following lemma will allow us to extend results concerning line segments to circular arcs.  Proposition \ref{prop:convexcorner} contains the key construction in proving Theorem \ref{thm:corner}.

\begin{lemma}\label{lem:chordAreaOrder}Suppose $C$ is a circle of radius $r$.  Then the smaller area determined by a chord of length $\ell$ has area  $O(\ell^3)$.
\end{lemma}

\begin{proof}
The area is given by $g(\ell):=\frac{1}{2} r^2 \left(2\arcsin\left(\frac{\ell}{2r}\right)-\sin\left(2\arcsin\left(\frac{\ell}{2r}\right)\right)\right)$.
 An easy calculation shows that $g(0)=\frac{dg}{d\ell}(0) = \frac{d^2 g}{d\ell^2}(0)= 0$, but $\frac{d^3g}{d\ell^3}(0) = \frac{1}{2r}\neq 0$.
\end{proof}

\begin{proposition}\label{prop:convexcorner}Suppose $Q\subseteq\mathbb{R}^2$ is a polygonal region and $K\in Q$ is a convex corner.  Suppose $S\subseteq Q$ and $\partial S$ contains exactly one smooth component containing $K$, which is either a line segment or an arc of a circle.  Then there exists $S' \subseteq Q$, an arbitrarily small deformation of $S$, with the following properties:

\begin{itemize}\item $|S| = |S'|$
\item $P(S) > P(S')$
\item No smooth component of $\partial S'$ contains $K$.
\end{itemize}

\end{proposition}

\begin{proof}
    We will first consider the case when the smooth boundary component containing $K$ is a line segment. For a small enough neighborhood of $K$, we can consider only the corner and this single line segment. Let $P$ be a point on $\partial S$ and let $\ell$ be the length of this line. Note that since $K$ is a convex corner, $\partial S$ makes at least one acute angle with $\partial Q$. Thus, without loss of generality, we have the following picture for a neighborhood of $K$.

    \begin{center}
        \begin{tikzpicture}[scale=0.8]
            \draw[ultra thick] (-4.5,2) -- (0,2);
            \draw[thick] (-4.5,2) -- (0,5);
            \filldraw (-4.5,2) circle (2pt) node[anchor=north east] {$K$};
            \draw (-2,3.4) node {$\ell$};
            \draw (-1.3,4.5) node {$\partial S$};
            \draw (-1.3,1.7) node {$\partial Q$};
            \draw (-3.8,2) arc (0:25:1);
            \draw (-3.8,2.3) node[anchor=west] {$\theta$};
            \filldraw (0,5) circle (2pt) node[anchor=north west] {$P$};
            \begin{scope}[shift={(6,0)}]
                \draw[ultra thick] (-4.5,2) -- (0,2);
                \draw[thick] (-4.5,2) -- (0,5);
                \filldraw (-4.5,2) circle (2pt) node[anchor=north east] {$K$};
                \filldraw (-3.2,2) circle (2pt) node[anchor=north west] {$A$};
                \filldraw (-3.2,3.5) circle (2pt) node[anchor=south east] {$B$};
                \draw (-3.8,2) node[anchor=north] {$\varepsilon$};
                \draw (-3.2,3.2) node[anchor=east] {$w$};
                \draw (-1.5,4.7) node {$\partial S'$};
                \draw (-1.3,1.7) node {$\partial Q$};
                \draw (-3.8,2) arc (0:25:1);
                \draw (-3.8,2.3) node[anchor=west] {$\theta$};
                \filldraw (0,5) circle (2pt) node[anchor=north west] {$P$};
                \draw[thick,dashed] (-3.2,2) -- (-3.2, 3.5) -- (0,5);
                \filldraw (-3.2,2.85) circle (2pt) node[anchor=north west] {$C$};
            \end{scope}
        \end{tikzpicture}
        
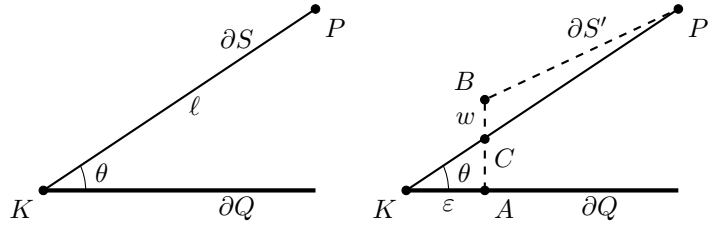
\captionof{figure}{On the left we have $\partial S$ intersecting a corner $K$. On the right we have a deformation $\partial S'$ which avoids corner $K$.}\label{fig:deformAwayFromCorner}
    \end{center}

    We will now deform $\partial S$ to $\partial S'$ with the desired properties. Consider a point $A$ which is $\varepsilon$ units away from the corner, and let $B$ be the point directly above $A$ so that $|S| = |S'|$. Let $C$ be the intersection point of $\overline{KP}$ and $\overline{AB}$, and let $w$ be the length of $\overline{BC}$. Note that in order for $|S| = |S'|$, we need the triangles $\Delta ACK$ and $\Delta BCP$ to have the same area. In other words, we need $$\frac{1}{2}\varepsilon^2 \tan\theta = \frac{1}{2}(\ell - \varepsilon\sec\theta)w\cos\theta.$$ This is true when
    \[
        w = \frac{\varepsilon^2\tan\theta}{\ell\cos\theta - \varepsilon}.
    \]
    Having achieved that $|S| = |S'|$, we now compare perimeters of the boundaries.  Note that $P(S) > P(S')$ if and only if $\left|\,\overline{KP}\,\right| > \left|\,\overline{AB}\,\right| + \left|\,\overline{BP}\,\right|.$
    As $|\overline{KP}|  = |\overline{KC}| + |\overline{CP}|$, it follows from the triangle inequality that this inequality is true if
    \begin{equation}\label{eq:lowerPerim}
        \left|\,\overline{KC}\,\right| > 
        2\left|\,\overline{BC}\,\right| + 
        \left|\,\overline{AC}\,\right|.
    \end{equation}
    Indeed, \begin{align*} |\overline{KP}| &= |\overline{KC}|+|\overline{PC}|\\
    & > 2|\overline{BC}| +|\overline{AC}| + |\overline{PC}|\\
    &= (|\overline{BC}| + |\overline{AC}|) + (|\overline{BC}| + |\overline{PC}|)\\
    & \geq |\overline{AB}| + |\overline{BP}|.\end{align*}
    
    Substituting into inequality \eqref{eq:lowerPerim} and rearranging, we find that inequality \eqref{eq:lowerPerim} holds if and only if 
    \begin{equation}\label{eq:lowerPerim2}
        2\varepsilon^2\tan\theta < \varepsilon(\ell - \varepsilon\sec\theta)(1-\sin\theta).
    \end{equation}
    Notice that since $\theta \in (0,\tfrac{\pi}{2})$, all factors in (\ref{eq:lowerPerim2}) are positive.  In addition, the left side is $O(\varepsilon^2)$ while the right side is $O(\varepsilon)$.  It follows that for all sufficiently small $\varepsilon$, the inequality \eqref{eq:lowerPerim2} holds. Hence, the proof is complete in the case where the smooth boundary component of $\partial S$ is a line segment.

    Now we will consider the case when the smooth boundary component containing $K$ is the an of a circle. By Proposition \ref{prop:circleTangentToSide}, we can assume that $\partial S$ does not touch $K$ tangentially to $Q$. Furthermore, since $K$ is a convex corner, $\partial S$ makes an acute angle at $K$ with at least one side of $\partial Q$. Thus, without loss of generality, we can assume that we have a neighborhood of $K$ which looks like one of the following, depending on the curvature.

    \begin{center}
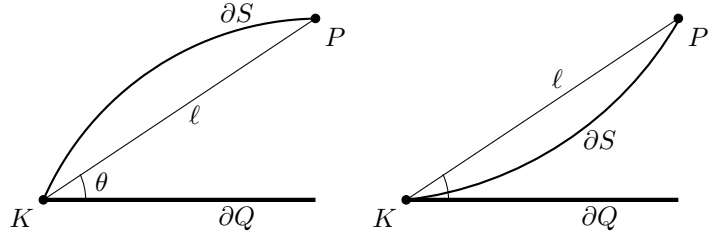

        \begin{tikzpicture}[scale=0.8]
            \draw[ultra thick] (-4.5,2) -- (0,2);
            \draw (-4.5,2) -- (0,5);
            \filldraw (-4.5,2) circle (2pt) node[anchor=north east] {$K$};
            \draw (-2,3.4) node {$\ell$};
            \draw (-1.3,5.1) node {$\partial S$};
            \draw (-1.3,1.7) node {$\partial Q$};
            \draw[thick] (-4.5,2) arc (156.4:91:5);
            \draw (-3.8,2) arc (0:25:1);
            \draw (-3.8,2.3) node[anchor=west] {$\theta$};
            \filldraw (0,5) circle (2pt) node[anchor=north west] {$P$};
        \begin{scope}[shift={(6,0)}]
            \draw[ultra thick] (-4.5,2) -- (0,2);
            \draw (-4.5,2) -- (0,5);
            \filldraw (-4.5,2) circle (2pt) node[anchor=north east] {$K$};
            \draw (-2,4) node {$\ell$};
            \draw (-1.3,3) node {$\partial S$};
            \draw (-1.3,1.7) node {$\partial Q$};
            \draw[thick] (-4.5,2) arc (-83.34:-30:6);
            \draw (-3.8,2) arc (0:25:1);
            \filldraw (0,5) circle (2pt) node[anchor=north west] {$P$};
        \end{scope}
        \end{tikzpicture}
        \captionof{figure}{Both possibilities for a circular arc to touch a corner acutely.}\label{fig:arcTouchingCorner}
    \end{center}
    
    Consider a point $P \in \partial S$, and let $\ell$ be the length of the secant line $\overline{KP}$. Now we can repeat the argument for the straight line case to create $\partial S'$ away from corner $K$. The only difference is that now $w$ will have to take into account the area, say $\delta$, between $\partial S$ and the secant line $\overline{KP}$. By doing this we can ensure that $|S| = |S'|$. Following the same logic as before, we will have that $P(S) > P(S')$ if
    \begin{equation}\label{eq:lowerPerim3}
        2\varepsilon^2\tan\theta  \pm \delta < \varepsilon(\ell - \varepsilon\sec\theta)(1-\sin\theta).
    \end{equation}
    Let $\varepsilon = \lambda \ell$. By Lemma \ref{lem:chordAreaOrder}, $\delta = O(\ell^3)$. Therefore, inequality \eqref{eq:lowerPerim3} is true if and only if there exists some $\lambda \in (0,1)$ and some $\ell > 0$ such that
    \[
        2\lambda^2\ell^2\tan\theta  \pm O(\ell^3) < \lambda\ell^2(1 - \lambda\sec\theta)(1-\sin\theta).
    \]
    We divide by $\ell^2$ and rearrange to see that inequality \eqref{eq:lowerPerim3} is true if and only if
    \[
        O(\ell) < \lambda(1-\sin\theta) - \lambda^2\sec\theta(1+\sin\theta).
    \]
    Let $\theta_{1}$ be the larger of the two angles formed by the tangent line at $K$ and the secant line at $K$. Then
    \[
        \lambda(1-\sin\theta_{1}) - \lambda^2\sec\theta_{1}(1+\sin\theta_{1}) < \lambda(1-\sin\theta) - \lambda^2\sec\theta(1+\sin\theta)
    \]
 for any $0\leq \theta < \theta_1$.   Thus, we have $P(S) > P(S')$ if
    \[
        O(\ell) < \lambda(1-\sin\theta_{1}) - \lambda^2\sec\theta_{1}(1+\sin\theta_{1}).
    \]
    We can choose $\lambda \in (0,1)$ so that the right-hand side is positive. Also, since the right-hand side is independent of $\ell$, we can then shrink $\ell$ to achieve our result. The circular arc case is done.
\end{proof}


Now we will move to the case when more than one smooth boundary component is touching a corner.

\begin{proposition}\label{prop:anycorner}Suppose $Q\subseteq\mathbb{R}^2$ is a polygonal region and $K\in Q$ is a corner.  Suppose $S\subseteq Q$ and $\partial S$ contains exactly two smooth components containing $K$ and that the interior angle of these two smooth components does not contain the corner sides at $K$.  Then there exists $S' \subseteq Q$, an arbitrarily small deformation of $S$, with the following properties:

\begin{itemize}\item $|S| = |S'|$
\item $P(S) > P(S')$
\item No smooth component of $\partial S'$ contains $K$.
\end{itemize}

\end{proposition}

\begin{proof}
    Since the interior angle these components make does not contain the corner sides, the angle within $Q$ is less than $\pi$. Hence, that angle bisected will be acute. Also, from Theorem \ref{thm:structure}, we have that these two smooth components of $\partial S$ which contain $K$ have the same curvature. Thus, we can simply bisect the interior angle and run the argument from Proposition \ref{prop:convexcorner} on each side to create an arbitrarily small deformation of $S$ with the desired properties. 
\end{proof}

We are ready to prove Theorem \ref{thm:corner}.

\begin{proof}[Proof of Theorem \ref{thm:corner}]  Let $K$ be a corner and let $Q_1$ and $Q_2$ denote the two sides of $Q$ containing $K$.  For both statements $(1)$ and $(2)$, we will prove the contrapositive.

$(1)$ Assume that $K$ is a convex corner.  Our goal is to show that if $\partial S$ contains at least one smooth boundary components containing $K$, then $S$ is not an isoperimetric minimizer.  If there is exactly one smooth component of $\partial S$ containing $K$, it makes an acute angle with at least one of $Q_1$ or $Q_2$.  Then the result follows from either Proposition \ref{prop:convexcorner} or Proposition \ref{prop:circleTangentToSide}, depending on the measure of this angle.

Thus, we assume there are at least two smooth boundary components of $\partial S$ containing $K$.  If we fix a side, say $Q_1$, of $Q$ containing $K$, then we observe that the angles at $K$ between $Q_1$ and each smooth component of $\partial S$ are almost distinct: if the smooth boundary components are segments, they must be distinct while if they are arcs of circles, at most two share an angle, see Figure \ref{fig:tangent}.  This follows because each smooth component of $\partial S$ has the same curvature. Hence, if they agree to first order and they have the same concavity relative to $Q_1$, then they are the same smooth boundary component.

\begin{center}
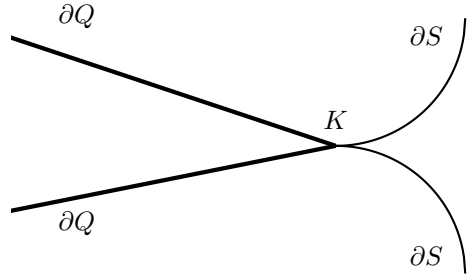

\centering
\begin{tikzpicture}
    \begin{axis}[
        axis lines = middle,
        ticks = none,
        xlabel = {},
        ylabel = {},
        xmin = -2.5, xmax = 1.5,
        ymin = -1.5, ymax = 1.5,
        axis line style = {draw=none}, 
        axis equal
    ]
        \addplot [
            domain=-3:0, 
            samples=2, 
            ultra thick, 
            black
        ] {-1/3 * x};
        
        \addplot [
            domain=-3:0, 
            samples=2, 
            ultra thick, 
            black
        ] {1/5 * x};
        
        \addplot [
            domain=0:1, 
            samples=100, 
            thick, 
            black
        ] {1 - sqrt(1 - x^2)};
        
        \addplot [
            domain=0:1, 
            samples=100, 
            thick, 
            black
        ] {sqrt(1 - x^2) - 1};

\draw (-2,1) node {$\partial Q$};
\draw (-2,-0.6) node {$\partial Q$};
\draw (0.7,0.85) node {$\partial S$};
\draw (0.7,-0.85) node {$\partial S$};
\draw (0,0.2) node {$K$};

    \end{axis}
\end{tikzpicture}
\captionof{figure}{Two smooth boundary components which are tangent at $K$.}\label{fig:tangent}
\end{center}

Since the angle is convex, any smooth boundary component must make an acute angle with at least one side of $Q$, say $Q_1$, at $K$.  Choose the smooth boundary component making the least angle with this side; if there are two, choose the circular arc which locally bends towards $Q_1$.  Now, we apply the proof of Proposition \ref{prop:convexcorner} or the proof of Proposition \ref{prop:circleTangentToSide} to this smooth boundary component. As long as this new smooth boundary component does not intersect the other smooth boundary components of $\partial S$, we obtain a region $S'$ with the same area but smaller perimeter.  

From the proof of Proposition \ref{prop:convexcorner}, the segment $w=\overline{CB}$ has length $O(\varepsilon^2)$.  Since the separation between different smooth components of $\partial S$ grows as $O(\varepsilon)$ and there are only finitely many smooth components, it follows that for $\varepsilon$ sufficiently small, the new boundary component of $S'$ does not intersect any of the other boundary components of $S$.

In addition, from the proof of Proposition \ref{prop:circleTangentToSide}, if we select $T$ such that the angle $\angle TKT'$ is smaller than the second smallest angle made at $K$, then the secant line between $K$ and $T$ does not intersect any other smooth boundary component of $\partial S$, and thus, none of the interpolating line segments does either. Therefore, in this case, we may form $S'$ in such a way that the new smooth boundary component does not intersect any of the other smooth boundary components of $S$. This completes the proof of $(1)$.

\bigskip

$(2)$  We will prove the contrapositive.  Assume that $\partial S$ contains at least two smooth boundary components containing $K$.  If any such smooth boundary component makes an acute angle with either of the two sides of $Q$ containing $K$, we repeat the proof of $(1)$ to show that $S$ is not an isoperimetric minimizer.  Thus, we may assume that all smooth boundary components make an obtuse or right angle with the two sides of $Q$ containing $K$.  It follows that any two smooth boundary components of $\partial S$ satisfy the hypothesis of Proposition \ref{prop:anycorner}.

As in the proof of (1), the angles made by any smooth boundary component of $\partial S$ and $Q_1$ are almost distinct.  Choose two such smooth boundary components for which the angle between them is minimized.  (If this minimum is achieved by multiple pairs of smooth boundary components, choose any pair of them.)  Now we apply the construction used in the proof of Proposition \ref{prop:anycorner}.  This completes the proof as long as the newly constructed smooth boundary components of $\partial S'$ do not intersect any of the other smooth boundary components of $\partial S$. By minimality of the angle, there are no smooth boundary components of $\partial S$ lying between the two modified smooth boundary components. And, since $w$ is $O(\varepsilon^2)$, we can argue as in the proof of $(1)$ to complete the proof.

\end{proof}

\section{\texorpdfstring{The isoperimetric profile of $Q_a$}{The isoperimetric profile of Qa}}\label{sec:allconnected}

For each $a\in [0,1)$, we let $Q_a := [0,1]^2 \setminus [0,a)^2$ be the unit square with a square of side length $a$ removed from the bottom left corner.  The goal of this section is to determine the isoperimetric profile for $Q_a$.  We begin by introducing four regions, each parameterized by the area $t$, which we label $S_1$, $S_2$, $S_3$, and $S_4$, see Figure \ref{fig:main}.

In Section \ref{sec:fat}, we introduce a technical lemma, Lemma \ref{lem:technical}, which will allow us to precisely define a function $f_a:[0,\tfrac{1}{2}|Q_a|]\rightarrow \mathbb{R}$, which we will eventually prove is the isoperimetric profile for $Q_a$.  In Section \ref{sec:Sconnected}, we use Lemma \ref{lem:technical} to prove that an isoperimetric minimizer must be connected and must have precisely one smooth boundary component. In Section \ref{sec:faiso}, we use Lemma~\ref{lem:technical} to show that $f_a$ is the relative isoperimetric profile for connected regions with precisely one smooth boundary component.  The proof of Lemma \ref{lem:technical} is postponed to Section~\ref{sec:IV}.

\begin{center}
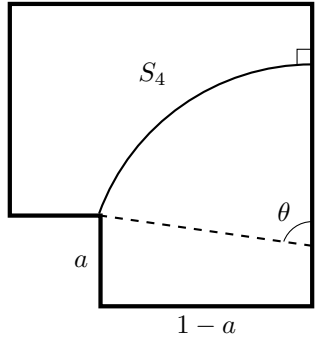

    \begin{tikzpicture}[scale=4]
        \draw[ultra thick] (0.3,0) -- (1,0) -- (1,1) -- (0,1) -- (0,0.3) -- (0.3,0.3) -- cycle;
        \draw[thick,dashed] (1,0.2) -- (0.3,0.3);
        \draw[thick] (1,0.8) arc (90:160:0.75);
        \draw (1,0.85) -- (0.95,0.85) -- (0.95,0.8);
        \draw (1,0.28) arc (90:160:0.1);
        \draw (0.96,0.25) node[anchor= south east] {$\theta$};
        \draw (0.29,0.15) node[anchor=east] {$a$};
        \draw (0.65,0) node[anchor=north] {$1-a$};
        \draw (0.55,0.7) node[anchor=south east] {$S_4$};
    \end{tikzpicture}
    \captionof{figure}{One minimizing region}\label{fig:L}
\end{center}

We define four regions $S_1,S_2,S_3,$ and $S_4$ as follows.   The region $S_1$ is any region bounded by a quarter circle ending on sides of $Q_a$ orthogonally, except that we do not allow a quarter circle surrounding the non-convex corner.  The region $S_2$ is bounded by a line segment of length $1$.  The region $S_3$ is bounded by a line segment of length $1-a$.  And finally, the region $S_4$ is bounded by a curve which touches the nonconvex corner, intersects a wall of length 1 perpendicularly, and does not make an acute angle with either side of length $a$ at the nonconvex corner. For fixed $a$, such a region is characterized by a single angle $\theta$ and a choice of the wall of length $1$ it intersects, see Figure~\ref{fig:L}.  For concreteness in proofs, we will always assume an $S_4$ region touches the wall on the right, i.e., for fixed $a$, the notation $S_4(\theta)$ unambiguously describes a subset of $Q_a$.  We will eventually show that any minimizer is a region $S_i$ with $i\in \{1,2,3,4\}$.

\subsection{The isoperimetric profile}\label{sec:fat}

In this section, we begin with a technical lemma whose proof will be postponed to Section \ref{sec:IV}.

\begin{lemma}\label{lem:technical}  All the following hold:

\begin{enumerate}\item  There is a number $\theta_{max}\in (0,\tfrac{\pi}{2})$ with the property that for any $a\in [0,1)$ and for any $t\in \left[a(1-a), \frac{1}{2}(1-a^2)\right]$ there is a unique $\theta = \theta(t)\in [0,\theta_{max}]$ for which $|S_4(\theta)| = t$.

\item  For each fixed $a\in[0,1)$,  expressing $\theta$ in terms of $t$, we have $$\frac{d}{dt} P(S_4) = \frac{\sin \theta(t)}{1-a} > 0 \quad \text{and}\quad \frac{d^2}{dt^2}P(S_4) > 0$$ for all $t\in \left[a(1-a), \frac{1}{2}(1-a^2)\right].$

\item  For each fixed $a\in [0,1)$, $P(S_4) < \sqrt{2}(1-a)$ for all $t\in\left[a(1-a), \frac{1}{2}(1-a^2)\right]$.

\item  There is a value $t_0\approx 0.48$ which has the following property:  for any $t\in [0,t_0]$, there is a unique pair $(a(t),\theta(t))$ for which $|S_4| = t$ and $P(S_4) = 1$, but for any $t > t_0$ no such solution exists.  Moreover, $a(t)$ is a strictly increasing function of $t$.

\item  There is a number $\beta \approx 0.23$ for which $P(S_4) < 1$ for all $a \in (\beta,1)$ and all $t\in \left[a(1-a), \frac{1}{2}(1-a^2)\right]$.

\item  There are numbers $\alpha \approx 0.10$ and $\gamma = \frac{1}{1+\pi}$ with the property that for any $a\in [\alpha,\gamma]$, there is a unique area $t= \sigma(a)$ for which $P(S_4) = \sqrt{\pi |S_4|}$.

\item  For $a\in [\alpha,\gamma]$, $\sigma$ is a strictly decreasing function.  In addition, $\sigma(\alpha) = \frac{1}{\pi}$ and $\sigma(\gamma) = \frac{1}{\pi}(1-\gamma)^2$.  
\end{enumerate}

\end{lemma}

The precise definitions of $\alpha$ and $\beta$ are that $\alpha = a(\tfrac{1}{\pi})$ and $\beta = a(t_0)$ in the context of Lemma \ref{lem:technical}(4).  The function $\tau = \tau(a)$ is defined to be the inverse of the function $t\mapsto a(t)$, from Lemma \ref{lem:technical}(4).

In terms of these parameters, recalling that $\frac{1}{2}|Q_a| = \frac{1}{2}(1-a^2)$, we define the following continuous piecewise-smooth function $f_a:\left[0, \frac{1}{2}(1-a^2)\right]\rightarrow \mathbb{R}$.  Lemma \ref{lem:technical}(4) and (6) indicate that it is well-defined.

\begin{equation}\label{eq:profile}
       f_a(t)= \begin{cases}
            0 \leq a \leq \alpha & \begin{cases}
                (\pi t)^{1/2} & \hspace{4.9em}0\leq t \leq \frac{1}{\pi} \\
                1 & \hspace{4.6em}\frac{1}{\pi} \leq t \leq \frac{1}{2}(1-a^2)
            \end{cases} \\[2em]
            \alpha \leq a \leq \beta & \begin{cases}
                (\pi t)^{1/2} & \hspace{4.7em}0\leq t \leq \sigma \\
                P(S_4) & \hspace{4.6em}\sigma \leq t \leq \tau \\
                1 & \hspace{4.65em}\tau \leq t \leq \frac{1}{2}(1-a^2)
            \end{cases} \\[2em]
            \beta \leq a \leq \gamma & \begin{cases}
                (\pi t)^{1/2} & \hspace{4.7em}0\leq t \leq \sigma \\
                P(S_4) & \hspace{4.6em}\sigma \leq t \leq \frac{1}{2}(1-a^2) \\
            \end{cases} \\[2em]
            \gamma \leq a < 1 & \begin{cases}
                (\pi t)^{1/2} & \hspace{4.7em} 0\leq t \leq \frac{1}{\pi}(1-a)^2 \\
                1-a & \hspace{1em}\frac{1}{\pi}(1-a)^2 \leq t \leq a(1-a) \\
                P(S_4) & \hspace{1.7em}a(1-a) \leq t \leq \frac{1}{2}(1-a^2)
            \end{cases}
        \end{cases}
\end{equation}

The formulas appearing in $f_a$ are of the form $P(S_i)$ for some $i\in \{1,2,3,4\}$.  In particular, $f_a$ is an upper bound for the isoperimetric profile of $Q_a$.

The rest of Section \ref{sec:allconnected} is focused on proving the following result.

\begin{theorem}\label{thm:iso} For each $a\in [0,1)$, $f_a$ is the isoperimetric profile for $Q_a$.

\end{theorem}

We conclude this subsection with a lemma providing a useful upper bound for the function $f_a$. To state it, we let $T = T(a)$ denote the $t$-value at which $f_a$ transitions from $(\pi t)^{1/2}$ to a different formula.  Thus,
\[
    T(a) =
    \begin{cases}
        \frac{1}{\pi} & 0\leq a \leq \alpha \\
        \sigma (a) & \alpha \leq a \leq \gamma \\
        \frac{1}{\pi}(1-a)^2 & \gamma \leq a < 1
    \end{cases}.
\]
We also note that an $S_1$ region has area bounded by $\min\left\{\frac{\pi}{4}, \frac{\pi}{2}(1-a)^2\right\}$.  This follows because the largest such region occurs when it has $(1,1)$ as its center and the radius is bounded above by $\min\{1, \sqrt{2}(1-a)\}$ to avoid hitting both the walls of side length $1-a$ and the non-convex corner at $(a,a)$.

\begin{lemma}\label{lem:ffacts}For each $a\in [0,1)$ and any $t\in \left[0,\min\left\{\frac{\pi}{4}, \frac{\pi}{2}(1-a)^2\right\}\right]$ we have $$(\pi t)^{1/2} \geq f_a(t).$$  In addition, this inequality is strict for any $t> T$.

\end{lemma}

\begin{proof} 
We first observe that the function $(\pi t)^{1/2}$ has a negative second derivative on its domain.  On the other hand, the constant functions $1$ and $1-a$ and the function $P(S_4)$ have non-negative second derivatives with respect to $t$ on their domains, as follows from Lemma \ref{lem:technical}(2).

Set $h(t):=(\pi t)^{1/2} - f_a(t)$.  To complete the proof it is sufficient to show $h(t)\geq 0$ for $t\in \left[T,\frac{1}{2}(1-a^2)\right]$ with equality if and only if $t = T$.  Observe that $h$ is strictly concave. Thus, $\{t:h(t)\geq 0\}$ is an interval and the zeros of $h$ can only occur at the endpoints.

By definition of $T$, $h(T) = 0$.  Thus, we need only show $h(t_1) > 0$, where $t_1$ is the $t$-value corresponding to the largest possible area of $S_1$.  As mentioned above, $t\leq \min\left\{\frac{\pi}{4}, \frac{\pi}{2}(1-a)^2\right\}$. However, the area is also constrained by the domain of $f_a$, that is $\left[0,\frac{1}{2}(1-a^2)\right]$.  Since $\frac{1}{2}(1-a^2) < \pi/4$ for all $a\in[0,1)$, we find that $t_1 = \min \{ \frac{1}{2}(1-a^2), \frac{\pi}{2}(1-a)^2\}$, and hence $t_1 = \frac{1}{2}(1-a^2)$ when $a\leq \frac{\pi-1}{\pi+1}\approx 0.52$ and $t_1 = \frac{\pi}{2}(1-a)^2$ when $a\geq \frac{\pi-1}{\pi+1}$.   We consider two cases.

\textit{Case 1}: Assume $a\leq \frac{\pi-1}{\pi+1}$. If $a\leq \beta$, then $f_a(t_1)\leq 1$ and $\sqrt{\pi t_1} \geq \sqrt{\frac{\pi}{2}(1-\beta^2)}$. Thus, $$h(t_1) \geq \sqrt{\frac{\pi}{2}(1-\beta^2)} - 1 > 0.$$    On the other hand, if $a\in \left(\beta, \frac{\pi-1}{\pi+1}\right]$, then we must verify that $\sqrt{\pi t_1} - P(S_4)>0$.  By Lemma \ref{lem:technical}(3), we see $P(S_4) < \sqrt{2}(1-a)$.  In addition, it is easy to see that the minimum value of $h$ on $[\beta, \frac{\pi-1}{\pi+1}]$ occurs at $\beta$.  Therefore, $$h(t_1) > \sqrt{\frac{\pi}{2}(1-a^2)} - \sqrt{2}(1-a) \geq \sqrt{\frac{\pi}{2}(1-\beta^2)} - \sqrt{2}(1-\beta) > 0,$$ completing the case where $a\leq \frac{\pi-1}{\pi+1}$.

\textit{Case 2}: Assume $a> \frac{\pi-1}{\pi+1}$.  Then $$h(t_1) = \sqrt{\pi \frac{\pi}{2}(1-a)^2} - P(S_4) > \frac{\pi}{\sqrt{2}}(1-a) - \sqrt{2}(1-a) > 0,$$ completing this case.

\end{proof}

\subsection{\texorpdfstring{The region $S$ must be connected}{The region S must be connected}}\label{sec:Sconnected}

In this section, we show that for an isoperimetric minimizer $S\subseteq Q_a$, the boundary $\partial S$ consists of exactly one smooth boundary component.  In particular, this implies that any minimizer must be connected.

 Recall that from Theorem \ref{thm:structure}, all smooth boundary components must have the same curvature.  Hence, either all smooth boundary components are line segments, or they are all arcs of circles of the same radius.  We begin with the case where they are line segments.

\begin{proposition}  If $S\subseteq Q_a$ is an isoperimetric minimizer for which all smooth boundary components of $\partial S$ are line segments, then there is only one smooth boundary component.
\end{proposition}

\begin{proof}  Suppose for a contradiction that $\partial S$ has at least two smooth line segment components.  Thus, $P(S) \geq 2(1-a)$. However, by Lemma \ref{lem:technical}(3), this is longer than $\max f_a$ when $a> \beta$, and longer than $1 = \max f_a$ when $a\leq \beta$.  Hence, $S$ is not an isoperimetric minimizer.
\end{proof}

We next move to the case where all smooth boundary components are arcs of a circle of radius $r$.  To that end, note that if such a smooth boundary component touches the nonconvex corner and it does not make an acute angle with either side, then one of three possibilities occurs:  it touches a side of length $a$ (so, must be a semicircle), it touches a side of length $1-a$ (so, it must be between quarter and half of a circle), or it touches a side of length 1 (so it bounds an $S_4$ region). The following lemma will be the key to handling the second case.

\begin{lemma}\label{lem:secondcase}
Suppose a circular segment of radius $r$ touches the non-convex corner in a non-acute angle and intersects a side of length $1-a$ perpendicularly.  Then $P\geq \sqrt{\pi t}$ where $P$ denotes the perimeter and $t$ denotes the area enclosed.
\end{lemma}

\begin{proof}
Suppose $\theta\in [\tfrac{\pi}{2},\pi]$ is the angle made by the endpoints of the circular segment and the center.  Then the area $t$ is given by $t = \tfrac{1}{2} r^2\theta -\tfrac{1}{2} r^2\cos\theta \sin\theta$ and the perimeter is $P = r\theta$.

Our goal is to show that $r\theta \geq \sqrt{\pi t}$, or, equivalently, that $\theta^2 \geq \frac{1}{2} \pi \theta - \frac{1}{2}\pi \cos\theta\sin\theta$.  Observe that when $\theta = \tfrac{\pi}{2}$, the inequality becomes the equality $\frac{\pi^2}{4} = \frac{\pi^2}{4}$.  Thus, it is sufficient to show the derivative of the left side is at least as large as the derivative of the right side for all $\theta \in [\tfrac{\pi}{2},\pi]$. The left side has derivative $2\theta \geq \pi$ on $[\tfrac{\pi}{2},\pi]$. On the other hand, the right side has derivative $\frac{1}{2}\pi - \frac{1}{2}\pi (\cos^2\theta - \sin^2\theta) = \pi \sin^2\theta$, which is bounded above by $\pi$ on $[\tfrac{\pi}{2},\pi]$.

\end{proof}

We are ready handle the case where all boundary components are arcs of a circle.

\begin{proposition} If $S\subseteq Q_a$ is an isoperimetric minimizer for which all smooth boundary components of $\partial S$ are cicular arcs of some radius $r$, then there is only one smooth boundary component.
\end{proposition}

\begin{proof} Suppose for a contradiction that $\partial S$ has at least two components and that all components have the same radius $r$. If the region $S$ lies inside one circular component and outside another, then by shrinking both at appropriate rates, the perimeter decreases while the area stays constant, contradicting minimality of $S$.  Thus, by replacing $S$ with its complement, we may assume without loss of generality that $S$ lies inside all of the circular components of its boundary.

Now, assume that no smooth boundary component bounds an $S_4$ region.  For a quarter, semi, or full circle, the length of the corresponding curve is at least $\sqrt{\pi t_i}$, where $t_i$ is the area enclosed by the $i$-th smooth boundary component.  For a curve touching the non-convex corner, it cannot make an acute angle with either side, by Theorem \ref{thm:corner}. Hence, it is either a semicircle touching the side of length $a$ or it touches the side of length $1-a$. In the first case, we find that its length is at least $\sqrt{\pi t}$, and the same conclusion holds for the second case by Lemma \ref{lem:secondcase}. Thus, in all cases, the length of $\partial S_i$ is bounded below by $\sqrt{\pi t_i}$.  Then, using Lemma \ref{lem:ffacts}, we find that
 $$P(S)  \geq \sum_i \sqrt{\pi t_i}  > \sqrt{\pi \sum_i t_i} = \sqrt{\pi t} \geq f_a(t),$$ proving that $S$ is not an isoperimetric minimizer.

It remains to consider the case where $\partial S$ contains a curve bounding an $S_4$ region. Such a curve has radius larger than or equal to $1-a$.  Observe that a quarter circle with radius larger than or equal to $1-a$ has perimeter at least $\frac{\pi}{2}(1-a)$.  Thus, an $S_4$ region with any other circular arc has perimeter at least $(1-a) + \frac{\pi}{2}(1-a) > \sqrt{2}(1-a)$.  From Lemma \ref{lem:technical}(3), this is larger than $\max f_a$ when $a\geq \beta$; and when $a < \beta$, $\sqrt{2}(1-a) > 1$, which is larger than $\max f_a$.

\end{proof}

\subsection{\texorpdfstring{Showing that $f_a$ is the isoperimetric profile}{Showing fa is the isoperimetric profile}}\label{sec:faiso}

In this section, we complete the proof of Theorem \ref{thm:iso}.  By the results of Section \ref{sec:Sconnected}, we already know that all isoperimetric minimizers are connected with exactly one smooth boundary component. Our first goal will be to show that minimizer must be a region of type $S_1,S_2,S_3$ or $S_4$. 

\begin{proposition}\label{prop:ruleoutlots}Suppose $S$ is an isoperimetric minimizer.  Then $S$ must be one of the regions $S_1,S_2,S_3,$ or $S_4$.
\end{proposition}

\begin{proof}We already know that $S$ must be connected with a single smooth boundary component.  We know from Theorem \ref{thm:corner} that $\partial S$ cannot intersect any corner of $Q_a$, except possibly the non-convex corner.  We also know from Theorem \ref{thm:structure} that $\partial S$ has constant curvature, hence, it is a line segment or a portion of a circle, and that any intersections between $\partial S$ and $\partial Q$ are orthogonal, except possibly if they intersect at the nonconvex corner.   In particular, if $\partial S$ is a line segment, then it is an $S_2$ or $S_3$ region. Thus, for the rest of the proof, we assume that $\partial S$ is an arc of a circle.

If $\partial S$ is a full circle or semicircle of area $t$, then $P(S) > \sqrt{\pi t} \geq f_a(t)$. Therefore, $S$ is not a minimizer. Thus, we may assume that either $\partial S$ is a quarter circle or that it touches the non-convex corner.

\textit{Case 1}:  $S$ is quarter circle.  Then $S$ is an $S_1$ region except when it is a quarter circle surrounding the non-convex corner. Assume it surrounds the non-convex corner. If there is a region of type $S_1$, which has the same area as $S$, then it has smaller perimeter than $S$, which would contradict minimality.  

If $a\in \left[0, 1 - \frac{\sqrt{2}}{\pi}\right]$, the largest quarter circle of type $S_1$ we can fit in $Q_a$ has radius $1$. By continuity, we can find a region of type $S_1$ with $|S_1| = |S|$, showing $S$ is not minimal.

If $a\in \left(1 - \frac{\sqrt{2}}{\pi},1\right)$, since $P(S) = \sqrt{\pi(t + a^2)}$, we see that $P(S)\geq \sqrt{\pi}a$. As $\pi > 2$ and $a > 1-a$ for $a > 1-\frac{\sqrt{2}}{\pi} \approx 0.54$, we have $P(S) \geq \sqrt{\pi} a > \sqrt{2} (1-a)$.  Finally, since $1-\frac{\sqrt{2}}{\pi} > \gamma$, we conclude that $P(S) >   \max f_a$, by Lemma \ref{lem:technical}(3). This contradicts minimality of $S$ and completes the first case.

\textit{Case 2}:  $S$ touches the non-convex corner.  We assume that $\partial S$ is a portion of a circle touching the non-convex corner of $Q_a$ and no other corner.  If the other end point of $\partial S$ is on a wall of length $a$, then $\partial S$ is a half-circle, so it is not a minimizer.  If it is on the wall of length $1$, it is an $S_4$ region.  Finally, if the other endpoint of $\partial S$ is on a wall of length $1-a$, then we reflect it as in Figure \ref{fig:reflection}, to obtain a new region $S'$ of area $t$ and with $|\partial S'| = |\partial S|$.  Then $\partial S'$ touches a wall of length $1-a$ and a wall of length $1$.  If it hits the wall of length $1$, it is a quarter circle, hence a region of type $S_1$.  If it does not hit perpendicularly, then $S'$ is not a minimizer by Theorem \ref{thm:structure}.

\begin{center}
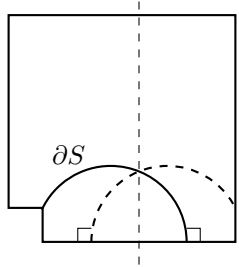

\begin{tikzpicture}[scale=3]
                \draw[thick] (0,1)  -- (1,1)  -- (1,0)  -- (0.15,0)  -- (0.15,0.15)  -- (0,0.15)  -- cycle;
                \draw[thick] (0.785,0) arc (0:155:0.335);
                \draw[thick,dashed] (0.365,0) arc (180:30:0.335);
                \draw[dashed] (0.575,-0.1) -- (0.575,1.08);
                \draw (0.15,0.3) node[anchor=south west] {$\partial S$};
                \draw (0.305,0) -- (0.305,0.06) -- (0.365,0.06);
                \draw (0.845,0) -- (0.845,0.06) -- (0.785,0.06);
\end{tikzpicture}
\captionof{figure}{Reflecting the region}\label{fig:reflection}
\end{center}

\end{proof}

We now work towards completing the proof of Theorem \ref{thm:iso} by determining for each $a\in [0,1)$ and each $t\in \left(0,\frac{1}{2}(1-a^2)\right]$ which type of the admissible regions has the shortest perimeter.  We recall that $T = T(a)$ is defined to be the $t$-value at which the formula for $f_a$ transitions away from $\sqrt{\pi t}$.  The next proposition indicates that for any $a\in [0,1)$, $f_a(t)$ is the isoperimetric profile for all $t\leq T$.

\begin{proposition}\label{prop:fsmallt} For any $a\in [0,1)$ and any $t\in (0,T]$, $S_1$ regions are isoperimetric minimizers.  Moreover, they are the only minimizers, unless $t = T$.
\end{proposition}

\begin{proof}

By Proposition \ref{prop:ruleoutlots}, we only need to compare $S_1$ regions with $S_2,S_3,$ and $S_4$ regions.  

\textit{Case 1}: Comparison with $S_2$ regions. Since $T\leq \frac{1}{\pi}$ (by Lemma \ref{lem:technical}(7)), it follows that $P(S_1) < P(S_2) = 1$ when $t<T$.

\textit{Case 2}: Comparison with $S_3$ regions.  We begin with the assumption that $a\in [\gamma,1)$.  Then $T(a) = \frac{1}{\pi}{(1-a)^2}$.  Observe that when $t= T$, $P(S_1) = P(S_3)$, and hence $P(S_1) < P(S_3)$ for $t < T$.  Next, when $a\in [0,\gamma)$, we note that there is an $S_3$ region only when $t\leq a(1-a)$.  Since $a< \gamma$, it follows that $a(1-a) < \frac{1}{\pi}(1-a)^2\leq T$, by Lemma \ref{lem:technical}(7). 

In particular, the proof of the proposition is complete for $t\in (a(1-a), T]$, since there is no corresponding $S_3$ region; and, for $t\in(0, a(1-a)]$, we have $$P(S_1) = \sqrt{\pi t}  \leq \sqrt{\pi a(1-a)} <  1-a = P(S_3).$$

\textit{Case 3}: Comparison with $S_4$ regions. We consider three subcases:  $a\in[0,\alpha)$, $a\in[\alpha,\gamma]$, and $a\in (\gamma,1)$.  

When $a\in [0,\alpha)$, by Lemma \ref{lem:technical}(4), the $t$-value at which $P(S_4(t)) = 1$ is less than $1/\pi$. Hence $P(S_1) < P(S_4)$ at $t = 1/\pi = T$, completing the case where $a\in [0,\alpha)$.

Next, assume $a\in [\alpha,\gamma]$.  Then, by definition of $\sigma$, when $t= T$ it follows that $P(S_1) = P(S_4)$.  The proof in this subcase will be complete once we show $\frac{d}{dt} P(S_1) > \frac{d}{dt} P(S_4)$ for all $t < T$.   Lemma \ref{lem:technical}(2) implies that 
\begin{equation*}
\frac{d}{dt} P(S_4)  = \frac{\sin \theta}{1-a} \leq \frac{1}{1-\gamma}\approx 1.32\quad\textrm{and}\quad\frac{d}{dt} P(S_1) = \frac{\sqrt{\pi}}{2\sqrt{t}} \geq \frac{\sqrt{\pi}}{2\sqrt{T}}.
\end{equation*}
Since $\sigma\leq \frac{1}{\pi}$, by Lemma \ref{lem:technical}(7), we have that $\frac{d}{dt} P(S_1) \geq \frac{\sqrt{\pi}}{2\sqrt{\frac{1}{\pi}}} = \frac{\pi}{2} \approx 1.57$.  Thus, $\frac{d}{dt} P(S_1) > \frac{d}{dt} P(S_4)$ for all $t< T$.

Finally, we assume $a\in (\gamma,1)$.  We follow a similar strategy to the proof in the previous sub-case.  First, since $a > \gamma$ and since a region of type $S_4$ has perimeter at least $1-a$, we see that $P(S_1) < P(S_4)$ at $t= T$. We have $\frac{d}{dt} P(S_4)\leq \frac{1}{1-a}$.  Substituting $t = T = \frac{1}{\pi}(1-a)^2$ into $\frac{d}{dt} P(S_1)$, we get $\frac{\pi}{2}\cdot\frac{1}{1-a}$.  Thus, $\frac{d}{dt} P(S_1) > \frac{d}{dt} P(S_4)$ when $a > \gamma$ as well.

\end{proof}

We can now complete the proof that $f_a$ is the isoperimetric profile of $Q_a$.

\begin{proposition}For any $a \in [0,1)$ and any $t \in \left(T, \frac{1}{2}(1-a^2)\right]$, $f_a$ is the isoperimetric profile for $Q_a$.

\end{proposition}

\begin{proof}From Lemma \ref{lem:ffacts}, since $t > T$ by hypothesis, we do not need to consider $S_1$ regions.  It remains to compare $S_2,S_3$, and $S_4$ regions. There are $S_3$ regions  only when $t\leq a(1-a)$ and there are $S_4$ regions only when $t \geq a(1-a)$ (and when $t = a(1-a)$, an $S_3$ region is an $S_4$ region). Hence, there is no $t$ for which we can compare $S_3$ and $S_4$.  Thus, we need only compare regions of type $S_2$ with regions of type $S_3$ or $S_4$.  We also note that, as shown in the proof of Proposition~\ref{prop:fsmallt}, for $a\leq \gamma$ we do not need to consider $S_3$ regions, as there is always an $S_1$ region with with the same area but smaller perimeter.  We break into cases depending on the value of $a$.

\textit{Case 1}:  $a\in[0, \alpha]$.  In the proof of Proposition \ref{prop:fsmallt}, we showed that the perimeter of an $S_4 = S_4(t)$ region is $1$ for some $t < \frac{1}{\pi}$.  Since $\frac{d}{dt} P(S_4)  >0$ by Lemma \ref{lem:technical}(2), it follows that $1 = P(S_2) < P(S_4)$ for all $t > \frac{1}{\pi}$.

\textit{Case 2}:  $a\in(\alpha,\beta]$.  By definition of $\sigma$, $P(S_4) = P(S_1)$ at $t=\sigma$.  Since $\sigma < \frac{1}{\pi}$ by Lemma \ref{lem:technical}(7), $P(S_1) < 1$, so $P(S_4) < 1 = P(S_2)$ at $\sigma$.  As $P(S_4)$ is an increasing function of $t$ by Lemma \ref{lem:technical}(2), and $\tau$ is defined as the $t$-value at which $P(S_4) = 1$,  it follows that $P(S_4) < 1= P(S_2)$ for $t\in(T,\tau)$, and that $P(S_4) > P(S_2)$ for $t < \tau$.

\textit{Case 3}: $a\in(\beta, \gamma)$.  By Lemma \ref{lem:technical}(5), $P(S_4) < 1 = P(S_2)$ for all $t\in \left[0,\frac{1}{2}(1-a^2)\right]$ for which both are defined.  In particular, for $a\in (\beta,\gamma]$, there is an $S_4$ region with area $t$ for all $t \in [T, \frac{1}{2}(1-a^2)]$, so that no $S_2$ region is a minimizer when $a\in(\beta,\gamma]$.

\textit{Case 4} $a\in (\gamma,1)$.  Lemma \ref{lem:technical}(5) implies $P(S_4)< P(S_2)$ whenever they are both defined.  Hence, no $S_2$ regions are minimizers; and $S_3$ regions have a shorter perimeter than corresponding $S_2$ regions when $t\in [T,a(1-a)]$.

\end{proof}

\section{A proof of Lemma \ref{lem:technical}}\label{sec:IV} In this section, we prove Lemma \ref{lem:technical}. We consider an $S_4$ region, which is determined by $a$ and $\theta$; hence we sometimes write $S_4(a,\theta)$. An elementary geometric argument gives the following formulas for the perimeter and area of $S_4$.

\begin{proposition}\label{prop:geo}The perimeter of $S_4 = S_4(a,\theta)$ is given by \begin{equation}\label{eqn:perim}P(S_4) = \frac{(1-a)\theta}{\sin \theta}\end{equation} and the area of $S_4$ is \begin{equation}\label{eqn:area}|S_4| = \frac{(1-a)^2\theta}{2\sin^2 \theta} - \frac{(1-a)^2}{2\tan \theta} + a(1-a).\end{equation}

\end{proposition}

Both equations \eqref{eqn:perim} and \eqref{eqn:area} have continuous extensions to $\theta = 0$ with $P(S_4) = 1-a$ and $|S_4| = a(1-a)$.  We will always use these extensions.

Recall that $|Q_a| = 1-a^2$. We are interested in $t\in (0, \tfrac{1}{2}(1-a^2)]$.  The following lemma indicates the domain of $\theta$ takes the form $[0,\theta_{max}]$, where $\theta_{max}$ is independent of $a$. The value of $\theta_{max}$ is approximately $1.21$.

\begin{lemma}\label{lem:thetamax}  The equation $|S_4| = \frac{1}{2}(1-a^2)$ admits a unique solution $\theta_{max}\in[0,\tfrac{\pi}{2}]$.

\end{lemma}

\begin{proof} The equation $|S_4| = \frac{1}{2}(1-a^2)$ is equivalent to the equation \begin{equation}\label{eqn:thetamax} 1 = \frac{\theta}{\sin^2\theta} - \frac{1}{\tan\theta}.\end{equation} As $\theta$ approaches $0$ and $\tfrac{\pi}{2}$, the right side of equation \eqref{eqn:thetamax} approaches $0 $ and $\tfrac{\pi}{2}$, respectively.  By the Intermediate Value Theorem, equation \eqref{eqn:thetamax} has a solution $\theta_{max}$.

The derivative of the right side of equation \eqref{eqn:thetamax} is $-2\csc^2 \theta \left(\theta \cot \theta-1\right)$.  We claim that $\theta\cot\theta-1 < 0$ on $(0,\frac{\pi}{2})$.  Indeed, this follows from the fact that its limiting value at $0$ is $0$ and that its derivative is $\frac{1}{\sin\theta}\left(\cos\theta - \frac{\theta}{\sin\theta}\right) < 0$.  Having established the claim, uniqueness of the solution to equation \eqref{eqn:thetamax} now follows.

\end{proof}

We now prove Lemma \ref{lem:technical}(1)

\begin{proposition}\label{prop:inverse} The number $\theta_{max}$ has the property that for any fixed $a\in [0,1)$ and for any $t\in \left[a(1-a), \frac{1}{2}(1-a^2)\right]$, there is a unique $\theta = \theta(t)\in [0,\theta_{max}]$ for which $|S_4(\theta)| = t$.

\end{proposition}

\begin{proof} To prove existence, we note that the formula for $|S_4|$ shows that it is continuous as a function of $\theta$ on $(0,\theta_{max})$.  Since its limit as $\theta$ approaches $0$ is $a(1-a)$ and its value at $\theta_{max}$ is $\frac{1}{2} (1-a^2)$, existence follows from the Intermediate Value Theorem.

For uniqueness, fixing $a$, we compute $\frac{d |S_4|}{d\theta} = \frac{dt}{d \theta} = \frac{(1-a)^2 (\sin\theta - \theta\cos\theta)}{\sin^3\theta}$.  We claim that this is positive on $(0,\theta_{max})$.  Indeed, the factor $\frac{(1-a)^2}{\sin^3 \theta}$ is positive and $\sin\theta - \theta\cos\theta$ is positive, because it has limiting value $0$ at $\theta = 0$ and derivative $\theta\sin\theta > 0$ on $(0,\pi)$.  As $\frac{d |S_4|}{d \theta}$ is positive, it follows that for fixed $a$, there is at most one solution to the equation $|S_4| =t$.
\end{proof}

In light of Proposition \ref{prop:inverse}, for fixed $a\in[0,1)$ and $\theta\in [0,\theta_{max}]$, the map $\theta\mapsto |S_4(\theta)|$ has an inverse.  When we write $S_4(t)$, we mean $S_4(\theta)$ precomposed with this inverse.  We now prove Lemma \ref{lem:technical}(2).

\begin{proposition}\label{prop:derivativebounds} For each fixed $a\in[0,1)$,  $\frac{d}{dt} P(S_4(t)) = \frac{\sin \theta(t)}{1-a} > 0$ and $\frac{d^2}{dt^2}P(S_4(t)) > 0$ for all $t\in \left[a(1-a), \frac{1}{2}(1-a^2)\right]$.

\end{proposition}

\begin{proof} The chain rule gives $\frac{d P(S_4(t))}{d t} = \frac{dP(S_4(\theta))}{d \theta} \cdot\frac{d\theta}{d t}$.  The first factor simplifies to $\frac{(1-a)(\sin\theta-\theta\cos\theta)}{\sin^2\theta}$ while the second factor is $\frac{1}{\frac{d t}{d \theta}} = \frac{\sin^3\theta}{(1-a)^2 (\sin\theta-\theta\cos \theta)}$.  Multiplying these together, we find that $\frac{d}{dt} P(S_4(t)) = \frac{\sin \theta}{1-a} > 0$, as claimed.

For the second derivative, the chain rule gives $\frac{d^2 P(S_4(t)) }{dt^2} = \frac{d}{d\theta}(\frac{dP(S_4(t))}{dt})\cdot \frac{d\theta}{dt}$.  The first factor is $\frac{\cos\theta}{1-a} > 0$ and the second factor is $\frac{1}{\frac{dt}{d\theta}}$, which is positive, by the proof of Proposition \ref{prop:inverse}.

\end{proof}

We now prove Lemma \ref{lem:technical}(3).

\begin{lemma}\label{lem:PLmax}For each fixed $a\in [0,1)$ and any $t\in\left[a(1-a), \frac{1}{2}(1-a^2)\right]$ we have that $P(S_4(t)) < \sqrt{2}(1-a)$. 

\end{lemma}

\begin{proof}Because $\frac{dP}{dt} = \frac{\sin\theta}{1-a} > 0$, the maximum value of $P(S_4(t))$ occurs when $t = \frac{1}{2}(1-a^2)$.  Because $\frac{d\theta}{dt} > 0$, this occurs when $\theta = \theta_{max}$.  Recalling that $\theta_{max}$ solves equation \eqref{eqn:thetamax}, we see that $\theta_{max}$ also solves the equation $\frac{\theta_{max}}{\sin\theta_{max}} = \sin\theta_{max} + \cos\theta_{max} = \sqrt{2}\sin\left(\tfrac{\pi}{4} + \theta_{max}\right)$.  Thus, $$P(S_4(t)) \leq (1-a) \frac{\theta_{max}}{\sin\theta_{max}}  = (1-a) \sqrt{2}\sin\left(\frac{\pi}{4} + \theta_{max}\right).$$  Since $\theta_{max}\in (\tfrac{\pi}{4},\tfrac{\pi}{2})$, $\sin(\tfrac{\pi}{4} + \theta_{max}) < 1$, so we find $P(S_4) < \sqrt{2}(1-a)$.

\end{proof}

We now prove Lemma \ref{lem:technical}(4).  We define $$t_0 = \frac{-2 \sin^2\theta_{max}+\left(2 \theta_{max} -\cos\theta_{max}\right) \sin\theta_{max}+\theta_{max}}{2 \theta_{max}^{2}}\approx 0.48.$$

\begin{proposition}\label{prop:uniquea} The number $t_0$ has the following property:  for any $t\in [0,t_0]$, there is a unique pair $(a(t),\theta(t))$ for which $|S_4(t)| = t$ and $P(S_4(t)) = 1$, but for any $t > t_0$ no such solution exists.  Moreover, $a(t)$ is a strictly increasing function of $t$. 

\end{proposition}

\begin{proof}
It is easy to see that $P(S_4)=1$ if and only if $a = 1-\frac{\sin\theta}{\theta}$. Hence, $\theta$ determines $a$ uniquely.  If we substitute this into equation \eqref{eqn:area}, we find that\begin{equation}\label{eqn:tintermsoftheta}t =\frac{-2 \sin^2 \left(\theta \right)+\left(2 \theta -\cos \theta \right) \sin\theta +\theta}{2 \theta^{2}}.\end{equation}

As $\theta$ approaches $0$, this expression limits to $0$, and at $\theta_{max}$ we obtain $t_0$.  Thus, existence of a solution for each $t$ in the interval follows from the Intermediate Value Theorem.

For uniqueness, we compute $$\frac{dt}{d\theta} = \frac{\left(\cos\! \left(\theta \right)+2 \sin\! \left(\theta \right)-\theta \right) \left(\sin\! \left(\theta \right)-\cos\! \left(\theta \right) \theta \right)}{\theta^{3}}.$$   Since both factors of the numerator are positive 
on $(0,\theta_{max}]$, $\frac{dt}{d\theta}$ is strictly increasing. Thus, there is at most one solution to equation \eqref{eqn:tintermsoftheta}.  Moreover, since $\frac{dt}{d\theta} $ is positive, it follows that for any $t > t_0$, there is no solution to equation \eqref{eqn:tintermsoftheta} with $\theta\in(0,\theta_{max}]$, which in turn implies there is no solution to the system $P(S_4) = 1$, $|S_4| = t$.

We conclude the proof by showing that $a(t)$ is a strictly increasing function of $t$.  Assume $t_1 < t_2$ with both $t_1$ and $t_2$ in the interval and let $(a_1,\theta_1)$ and $(a_2,\theta_2)$ solve the above system with $t= t_1$ and $t=t_2$, respectively.   Then, from equation \eqref{eqn:tintermsoftheta} and from the fact that $\frac{dt}{d\theta} > 0$, we conclude that $\theta_1<\theta_2$.  Finally, from the equation $a = 1-\frac{\sin\theta}{\theta}$, we conclude that $a_1 < a_2$.

\end{proof}

We now obtain Lemma \ref{lem:technical}(5) as a corollary.  Recall that $\beta\approx 0.23$ is defined as $a(t_0)$ in the notation of the previous proposition.

\begin{corollary}The number $\beta$ has the property that $P(S_4(t)) < 1$ for all $a \in (\beta,1)$ and all $t\in \left[a(1-a), \frac{1}{2}(1-a^2)\right]$.
\end{corollary}

\begin{proof}
By definition, $t_0$ is the value obtained from equation \eqref{eqn:tintermsoftheta} when $\theta = \theta_{max}$. Hence, $\beta  =  1-\frac{\sin \theta_{max}}{\theta_{max}}\approx 0.23$. Equivalently, $\frac{(1-\beta)\theta_{max}}{\sin\theta_{max}} = 1$.  Then, for any $a > \beta$ and any $\theta \in (0,\theta_{max}]$, we see that $$P(S_4(a,\theta)) = \frac{(1-a)\theta}{\sin \theta} < \frac{(1-\beta)\theta_{max}}{\sin\theta_{max}} = 1.$$
\end{proof}

We now prove Lemma \ref{lem:technical}(6).  Recall that $\alpha\approx 0.1$ is defined to be $a(\tfrac{1}{\pi})$ in the notation of Proposition \ref{prop:uniquea}, while $\gamma = \frac{1}{1+\pi}\approx 0.24$.

\begin{proposition}\label{prop:sigmadefined} For any $a\in [\alpha,\gamma]$, there is a unique area $\sigma(a) = t$ for which $P(S_4(t)) = \sqrt{\pi |S_4(t)|}$.

\end{proposition}

\begin{proof}
Fix $\alpha \leq a\leq \gamma$.  Consider the function $j:[0,\theta_{max}]\rightarrow \mathbb{R}$ defined by \begin{equation}\label{eqn:j}j(\theta) = \frac{1}{1-a}\left(P^2(S_4) - \pi |S_4|\right).\end{equation}  Observe that a zero of $j$, when substituted into equation \eqref{eqn:area}, yields $\sigma(a)$. We compute that $j(0) = (1-a) - \pi a = (1-(\pi+1)a) \geq 0$, since $a\leq \gamma= \frac{1}{1+\pi}$. We claim that $j(\theta_{max}) < 0$ for $ a\geq \alpha$.  To see this, assume for a contradiction that $j(\theta_{max}) \geq 0$.  Solving this inequality for $a$, we find that $$a\leq -\frac{\pi\theta_{max} \tan  \theta_{max}  -2 \theta_{max}^{2} \tan\theta_{max}-\pi\sin^2 \theta_{max}}{2\pi \tan\theta_{max}  \sin^2 \theta_{max} -\pi\theta_{max}\tan \theta_{max} +2 \theta_{max}^{2} \tan\theta_{max}+\pi\sin^2 \theta_{max}}.$$  As the right side evaluates to approximately $0.03 < \alpha$, we have a contradiction.  By the Intermediate Value Theorem, for each $a\in [\alpha,\gamma]$, there is a number $\theta = \theta(a)$ for which $j(\theta) = 0$,  which gives $P(S_4(\theta)) = \sqrt{\pi |S_4(\theta)|}$. We claim that $\theta$ is unique.  Indeed, we have $$\frac{dj}{d\theta} = (1 - a) (\pi - 2 \theta) (-1 + \theta \cot\theta) \csc^2(\theta).$$ In the proof of Lemma \ref{lem:thetamax}, we showed $-1 + \theta\cot(\theta)$ is non-zero, and $\pi - 2\theta = 0$ if and only if $\theta = \frac{\pi}{2}$.  It follows that $\frac{dj}{d\theta}\neq 0$  on $[0,\theta_{\max}]$. Hence $j$ is injective.

\end{proof}

We conclude with a proof of Lemma \ref{lem:technical}(7).

\begin{lemma}\label{lem:sigmataubound}  For $a\in [\alpha,\gamma]$, $\sigma$ is a strictly decreasing function.  In addition, $\sigma(\alpha) = \frac{1}{\pi}$ and $\sigma(\gamma) = \frac{1}{\pi}(1-\gamma)^2$.  
\end{lemma}

\begin{proof}  We first claim that $\sigma(\alpha) = \frac{1}{\pi}$.  Indeed, by definition of $\alpha$, there is an angle $\theta_\alpha$ for which $|S_4(\alpha,\theta_{\alpha})| = \frac{1}{\pi}$ and $P(S_4(\alpha,\theta_{\alpha})) = 1$.  Observe that $P(S_4) = \sqrt{\pi |S_4|}$. Hence, by definition of $\sigma$, $\sigma(a) = \frac{1}{\pi}$.

We claim that $\sigma(\gamma) = \frac{1}{\pi}(1-\gamma)^2$.  To see this, observe that equation \eqref{eqn:area} is solved with $a = \gamma$ and $\theta = 0$.  Substituting these values into equations \eqref{eqn:perim} and \eqref{eqn:area}, we find that $P(S_4) = (1-\gamma)$ and $|S_4| = \gamma(1-\gamma) = \frac{1}{\pi}(1-\gamma)^2$.  Thus, $P(S_4)^2 = \pi |S_4|$, which implies $\sigma(\gamma) = |S_4| = \frac{1}{\pi}(1-\gamma)^2$.

Next, we show that $\sigma$ is a strictly decreasing function of $a$.   First, we use equation \eqref{eqn:j} with $j=0$ to define $\theta$ in terms of $a$.  We get that $$\frac{d\theta}{da} = -\frac{\frac{\partial j}{\partial a}}{\frac{\partial j}{\partial \theta}} = \frac{(-2\pi \sin^2\theta - \pi\cos\theta\sin\theta + \theta(\pi - 2\theta))\sin\theta}{2(-1 + a)(\pi - 2\theta)(\theta\cos\theta - \sin\theta)}.$$  We also define $t = |S_4|$ in terms of $\theta$ via equation \eqref{eqn:area}, as done in Proposition \ref{prop:inverse}. Hence $\frac{dt}{d\theta} = -(1-a)^2 \csc^2 \theta (\theta \cot\theta-1)$. Then $\sigma$ is defined by $\sigma(a) = t(\theta(a))$. By the chain rule, $$\frac{d\sigma}{da} = \frac{dt}{da} = \frac{dt}{d\theta}\cdot\frac{d\theta}{da} = \frac{(1 - a)(\theta(\pi - 2\theta)\csc^2\theta - \pi\cot\theta - 2\pi)}{2\pi - 4\theta}.$$

We now show $\frac{d\sigma}{da} < 0$.  First, because $\theta \in [0,\theta_{max}]\subseteq [0,\tfrac{\pi}{2})$, the denominator is positive.  Also, the factor $(1-a)$ is positive. Hence, we only need to verify that $(\theta(\pi - 2\theta)\csc^2\theta - \pi\cot\theta - 2\pi)<0.$  This factor limits to $-2-2\pi < 0$ as $\theta$ approaches $0$, and it has value $-2\pi < 0$ at $\theta = \pi/2$. We will show its derivative is non-zero. The derivative is $-2(\pi-2\theta)(\theta\cot\theta-1)\csc^2(\theta)$. The factors $(\pi-2\theta)$ and $\csc^2\theta$ are non-zero on $(0,\tfrac{\pi}{2})$ and the factor $(\theta\cot\theta - 1)$ was shown to be non-zero in the proof of Lemma \ref{lem:thetamax}. The proof is complete.

\end{proof}

\section{Isoperimetric Profile Pictures}

Here we present the graphs of the isoperimetric profile function $f_a$ defined by equation \eqref{eq:profile} for different values of $a$.
\begin{itemize}
    \item $0 \leq a \leq \alpha$.
    \begin{center}
        \begin{tikzpicture}[scale=3]
            \draw[->,thick] (-0.1,0) -- (1.2,0) node[anchor= west] {$t$};
            \draw[->,thick] (0,-0.1) -- (0,1.2) node[anchor=south] {$P(S)$};
            \draw[ultra thick,domain=0:1] plot(0.637*\x*\x,\x);
            \draw[ultra thick] (0.637,1) -- (1,1);
            \draw[thick,dashed] (0.637,1) -- (0.637,0) node[anchor=north] {$\frac{1}{\pi}$};
            \draw[thick,dashed] (1,1) -- (1,0) node[anchor=north] {$\scriptstyle \frac{1}{2}(1-a^2)$};
            \draw[thick,dashed] (0.637,1) -- (0,1) node[anchor=east] {$1$};
        \end{tikzpicture}
    \end{center}
    \item $\alpha \leq a \leq \beta$
    \begin{center}
        \begin{tikzpicture}[scale=3]
            \draw[->,thick] (-0.1,0) -- (1.2,0) node[anchor= west] {$t$};
            \draw[->,thick] (0,-0.1) -- (0,1.2) node[anchor=south] {$P(S)$};
            \draw[ultra thick,domain=0:0.9] plot(0.637*\x*\x,\x);
            \draw[ultra thick] (0.79,1) -- (1,1);
            \draw[thick,dashed] (0.51,0.9) -- (0.51,0) node[anchor=north] {$\sigma$};
            \draw[thick,dashed] (0.79,1) -- (0.79,0) node[anchor=north] {$\tau$};
            \draw[thick,dashed] (1,1) -- (1,0);
            \draw[thick,dashed] (0.79,1) -- (0,1) node[anchor=east] {$1$};
            \draw[ultra thick] (0.51,0.9) parabola (0.79,1);
            \draw[thick] (0.637,0.05) -- (0.637,-0.05) node[anchor=north] {$\frac{1}{\pi}$};
            \draw[thick,dashed] (1,1) -- (1,0) node[anchor=north] {$\scriptstyle \frac{1}{2}(1\text{-}a^2)$};
        \end{tikzpicture}
        \hspace{1em}
        \begin{tikzpicture}[scale=3]
            \draw[->,thick] (-0.1,0) -- (1.2,0) node[anchor= west] {$t$};
            \draw[->,thick] (0,-0.1) -- (0,1.2) node[anchor=south] {$P(S)$};
            \draw[ultra thick,domain=0:0.8] plot(0.637*\x*\x,\x);
            \draw[ultra thick] (0.9,1) -- (1,1);
            \draw[thick,dashed] (0.4,0.8) -- (0.4,0) node[anchor=north] {$\sigma$};
            \draw[thick,dashed] (0.9,1) -- (0.9,0) node[anchor=north] {$\tau$};
            \draw[thick] (0.637,0.05) -- (0.637,-0.05) node[anchor=north] {$\frac{1}{\pi}$};
            \draw[thick,dashed] (1,1) -- (1,0);
            \draw[thick,dashed] (0.9,1) -- (0,1) node[anchor=east] {$1$};
            \draw[ultra thick] (0.4,0.8) parabola (0.9,1);
        \end{tikzpicture}
    \end{center}
    \item $\beta \leq a \leq \gamma$
    \begin{center}
        \begin{tikzpicture}[scale=3]
            \draw[->,thick] (-0.1,0) -- (1.2,0) node[anchor= west] {$t$};
            \draw[->,thick] (0,-0.1) -- (0,1.2) node[anchor=south] {$P(S)$};
            \draw[ultra thick,domain=0:0.75] plot(0.637*\x*\x,\x);
            \draw[thick,dashed] (0.36,0.75) -- (0.36,0) node[anchor=north] {$\sigma$};
            \draw[thick,dashed] (1,1) -- (1,0) node[anchor=north] {$\scriptstyle \frac{1}{2}(1-a^2)$};
            \draw[thick,dashed] (1,1.08) -- (0,1.08) node[anchor=east] {$1$};
            \draw[thick,dashed] (0.36,0.75) -- (0,0.75) node[anchor=east] {$(\pi \sigma)^{1/2}$};
            \draw[ultra thick] (0.36,0.75) parabola (1,1);
        \end{tikzpicture}
    \end{center}
    \item $\gamma \leq a \leq 1$
    \begin{center}
        \begin{tikzpicture}[scale=2.8]
            \draw[->,thick] (-0.1,0) -- (1.4,0) node[anchor= west] {$t$};
            \draw[->,thick] (0,-0.1) -- (0,1.2) node[anchor=south] {$P(S)$};
            \draw[ultra thick,domain=0:0.6] plot(0.955*\x*\x,\x);
            \draw[ultra thick] (0.344,0.6) -- (0.72,0.6);
            \draw[thick,dashed] (0.344,0.6) -- (0.344,0) node[anchor=north] {$\scriptstyle\frac{1}{\pi}(1-a)^2$};
            \draw[thick,dashed] (0.72,0.6) -- (0,0.6) node[anchor=east] {$1-a$};
            \draw[thick,dashed] (0.72,0.6) -- (0.72,0) node[anchor=north] {$\scriptstyle a(1-a)$};
            \draw[ultra thick] (0.72,0.6) parabola (1.26,0.9);
            \draw[thick,dashed] (1.26,0.9) -- (1.26,0) node[anchor=north] {$\scriptstyle \frac{1}{2}(1-a^2)$};
        \end{tikzpicture}
        \hspace{0.5em}
        \begin{tikzpicture}[scale=2.8]
            \draw[->,thick] (-0.1,0) -- (1.3,0) node[anchor= west] {$t$};
            \draw[->,thick] (0,-0.1) -- (0,1.2) node[anchor=south] {$P(S)$};
            \draw[ultra thick,domain=0:0.3] plot(0.955*\x*\x,\x);
            \draw[ultra thick] (0.086,0.3) -- (0.63,0.3);
            \draw[thick,dashed] (0.086,0.3) -- (0.086,0) node[anchor=north] {$\scriptstyle\phantom{==}\frac{1}{\pi}(1-a)^2$};
            \draw[thick,dashed] (0.086,0.3) -- (0,0.3) node[anchor=east] {$1-a$};
            \draw[thick,dashed] (0.63,0.3) -- (0.63,0) node[anchor=north] {$\scriptstyle a(1-a)$};
            \draw[ultra thick] (0.63,0.3) parabola (0.765,0.5);
            \draw[thick,dashed] (0.765,0.5) -- (0.765,0);
        \end{tikzpicture}
    \end{center}
\end{itemize}

\bibliographystyle{plain}
\bibliography{biblio}

\end{document}